\documentclass[10pt,a4paper]{amsart}

\usepackage{graphicx,amssymb,color,amscd}
\usepackage{youngtab}
\usepackage[all]{xy}
\SelectTips{cm}{}

\theoremstyle{plain}
\newtheorem{theorem}{Theorem}[section]
\newtheorem*{theorem*}{Theorem}
\newtheorem{lemma}[theorem]{Lemma}
\newtheorem{proposition}[theorem]{Proposition}
\newtheorem{corollary}[theorem]{Corollary}
\theoremstyle{definition}
\newtheorem{definition}[theorem]{Definition}
\theoremstyle{remark}
 
\newtheorem{remark}[theorem]{Remark}
\newtheorem{example}[theorem]{Example}

\newtheorem*{acknowledgments}{Acknowledgments}
\numberwithin{equation}{section}
\numberwithin{figure}{section}

\newcommand{\bd}{\begin{description}}   
\newcommand{\ed}{\end{description}} 
\newcommand{\ba}{\begin{array}}      \newcommand{\ea}{\end{array}} 
\newcommand{\bc}{\begin{center}}     \newcommand{\ec}{\end{center}} 
\newcommand{\be}{\begin{enumerate}}  \newcommand{\ee}{\end{enumerate}} 
\newcommand{\beq}{\begin{eqnarray}}  \newcommand{\eeq}{\end{eqnarray}} 
\newcommand{\beQ}{\begin{eqnarray*}} \newcommand{\eeQ}{\end{eqnarray*}} 
\newcommand{\bi}{\begin{itemize}}    \newcommand{\ei}{\end{itemize}}

\newcommand\inv{\mathrm{Inv}}
\def\co{\colon\thinspace}

\def\ker{\mathop{\mathrm{Ker}}\nolimits}

\begin{document} 
\title{Riordan trees and the homotopy $sl_2$ weight system} 

\author[J.B. Meilhan]{Jean-Baptiste Meilhan} 
\address{Univ. Grenoble Alpes, IF, F-38000 Grenoble, France}
         \email{jean-baptiste.meilhan@ujf-grenoble.fr}
\author[S. Suzuki]{Sakie Suzuki} 
\address{The Hakubi Center for Advanced Research/Research Institute for Mathematical Sciences, Kyoto University, Kyoto, 606-8502, Japan. }
         \email{sakie@kurims.kyoto-u.ac.jp}
         
\begin{abstract} 
The purpose of this paper is twofold. 
On one hand, we introduce a modification of the dual canonical basis for invariant tensors of the $3$-dimensional irreducible representation of $U_{q}(sl_{2})$, 
given in terms of Jacobi diagrams, a central tool in quantum topology. 
On the other hand, we use this modified basis to study the so-called homotopy $sl_{2}$ weight system, 
which is its restriction to the space of Jacobi diagrams labeled by distinct integers. 
Noting that the $sl_{2}$ weight system is completely determined by its values on trees, 
we compute the image of the homotopy part on connected trees in all degrees;  
the kernel of this map is also discussed.  
\end{abstract} 

\maketitle
%

\section{Introduction}
The $sl_2$ weight system $W$ is a $\mathbb{Q}$-algebra homomorphism from the space $\mathcal{B}(n)$ of Jacobi diagrams labeled by $\{1,\ldots,n\}$ 
to the algebra $\inv\left(S(sl_2)^{{\otimes n}}\right)$ of invariant tensors of the symmetric algebra $S(sl_2)$. 
The relevance of this construction lies in low dimensional topology. 
Jacobi diagrams form the target space for the Kontsevich integral $Z$, which is universal among finite type and quantum invariants of knotted objects : 
in particular, by postcomposing $Z$ with the $sl_2$ weight system 
and specializing each factor at some finite-dimensional representation of quantum group $U_{q}(sl_{2})$, one recovers the colored Jones polynomial. 
Hence, while the results of this paper are purely algebraic, we will see that they are motivated by, 
and have applications to, quantum topology -- see Remark \ref{rem:topopipo} at the end of this introduction. 

An easy preliminary observation on the $sl_2$ weight system is the following.
\begin{lemma}\label{prop:connected}
 The $sl_2$ weight system is determined by its values on connected trees, i.e. connected and simply connected Jacobi diagrams.  
\end{lemma}
\noindent (Although this result might be well-known, a proof is given in Section \ref{sec:Bsl2}.)  

In this paper, we focus on the \emph{homotopy part} $\mathcal{B}^h(n)$, which is generated by diagrams labeled by distinct elements in $\{1,\ldots,n\}$. 
Here, the terminology alludes to the link-homotopy relation on (string) links, which is generated by self crossing changes. 
It was shown by Habegger and Masbaum \cite{HMa} that the restriction of the Kontsevich integral to $\mathcal{B}^h(n)$ 
is a link-homotopy invariant, and is deeply related to Milnor link-homotopy invariants, which are classical invariants generalizing the linking number. 

Let us state our main results on the \emph{homotopy $sl_2$ weight system}, 
that is, the restriction of the $sl_2$ weight system to $\mathcal{B}^h(n)$.  
Owing to Lemma \ref{prop:connected}, we can fully understand this map by studying the restrictions
 $$ W^{h}_n\co \mathcal{C}_{n} \to \inv(sl_{2}^{{\otimes n}})$$ 
of the $sl_2$ weight system to the space $\mathcal{C}_{n}$ of connected trees with $n$ univalent vertices labeled by distinct elements in $\{1,\ldots,n\}$.
Here, the target space $\inv(sl_{2}^{{\otimes n}})$ is the invariant part of the $n$-fold tensor power of the adjoint representation (the $3$-dimensional irreducible representation) of $sl_{2}$. 
Recall that the dimension of $\mathcal{C}_{n}$ is given by $(n-2)!$, 
while the dimension of $\inv(sl_{2}^{{\otimes n}})$ is known to be the so-called \cite{B} Riordan numbers $R_n$ 
which can be defined by $R_2=R_3=1$ and $R_n= (n-1)(2R_{n-1}+3R_{n-2})/(n+1)$. 
These numbers are also found under the name of Motzkin sums, or ring numbers in the literature. 

More generally, we have:
\begin{theorem}\label{1}
\begin{itemize}
\item[\rm{(i)}]
The weight system map $W^{h}_n$ is injective if and only if $n\le 5$. 
\item[\rm{(ii)}]
For $n$ odd and $n=2$, the weight system map $W^{h}_n$ is surjective.  
\item[\rm{(iii)}]
For $n\geq 4$ even,  $W^{h}_n$ has a $1$-dimensional cokernel, spanned by $c^{\otimes n}$, where  $c=\frac{1}{2}h\otimes h +e\otimes f+f\otimes e \in \inv (sl_{2}^{\otimes 2})$. 
\end{itemize}
\end{theorem}
The dimensions of $\mathcal{C}_n$, $\inv (sl_2^{\otimes n})$ and $\ker W^{h}_n$ are given in Table \ref{table1}.  
\begin{table}[h!]
\begin{center}
\begin{tabular}{|c|c|c|c|c|c|c|c|c|c|c|}
\hline
 $n$                            & 2 & 3 & 4 & 5 & 6   & 7     & 8     & 9 & $k$\\ 
\hline
$\dim \mathcal{C}_n$   & 1 & 1 & 2 & 6 & 24 & 120 & 720 &  5040 & $(k-2)!$ \\
\hline
$\dim \inv (sl_2^{\otimes n})$     & 1 & 1 & 3 & 6 & 15 & 36   &  91 & 232 & $R_k$\\
\hline
$\dim \ker W^{h}_n$     & 0 & 0 & 0 & 0 & 10 & 84   &  630 & 4808 & $(k-2)! - R_k + \frac{1+ (-1)^k}{2}$\\[0.1cm]
\hline
\end{tabular}
\end{center}
\caption{The dimensions of $\mathcal{C}_n$, $\inv (sl_2^{\otimes n})$ and $\ker W^{h}_n$.} \label{table1}
\end{table}

Let $\mathfrak{S}_{n}$ be the symmetric group in $n$ elements.
The spaces $\mathcal{C}_{n}$ and $\mathrm{Inv}(sl_{2}^{\otimes n})$ have  $\mathfrak{S}_{n}$-module structures, such that $\mathfrak{S}_{n}$ acts on $\mathcal{C}_{n}$  by permuting the labels,
and acts  on  $\mathrm{Inv}(sl_{2}^{\otimes n})$ by permuting the factors.
The $sl_{2}$ weight system is a  $\mathfrak{S}_{n}$-module homomorphism, and the characters  $\chi_{\mathcal{C}_{n}}$ and $\chi_{\mathrm{Inv}(sl_{2}^{\otimes n}) }$ are already known (see Lemma \ref{sn} and Proposition \ref{Kon}).
Thus, by Theorem \ref{1}, we can determine the character $\chi_{\mathrm{ker} (W_n^h)}$ of the kernel of $W_{n}^{h}$ as follows.
\begin{corollary}\label{cha}
\rm{(i)} For $n=2$ or  $n> 2$ odd, we have
$$\chi_{\mathrm{ker} (W_n^h)}=\chi_{\mathcal{C}_{n}}-\chi_{\mathrm{Inv}(sl_{2}^{\otimes n})}
\quad\textrm{and} \quad\chi_{\mathrm{Im} (W_n^h)}=\chi_{\mathrm{Inv}(sl_{2}^{\otimes n}) }.$$
\rm{(ii)} For $n\geq 4$ even, we have
$$\chi_{\mathrm{ker} (W_n^h)}=\chi_{\mathcal{C}_{n}}-\chi_{\mathrm{Inv}(sl_{2}^{\otimes n})}+ \chi_{U}
\quad\textrm{and} \quad 
\chi_{\mathrm{Im} (W_n^h)}=\chi_{\mathrm{Inv}(sl_{2}^{\otimes n}) }- \chi_{U},$$
where $U$ is the trivial representation.
\end{corollary}

Although the proof of Theorem \ref{1} is mainly combinatorial, it heavily relies on the following algebraic result.
\begin{theorem*}[Theorem \ref{JBbasis}]
The set 
 $$ \mathfrak{I}_n := \{  W(T) \textrm{ ;  $T$ is a Riordan tree of order $n$} \} $$
forms a basis for $\inv(sl_2^{\otimes n})$. 
\end{theorem*} 
Here, Riordan trees of order $n$ are a special class of elements of $\mathcal{B}^h(n)$ ; 
roughly speaking, a Riordan tree is a disjoint union of linear tree diagrams (i.e. of the shape of Figure \ref{fig:mTm}), whose label sets comprise a Riordan partition of $\{1,\ldots,n\}$ -- see Definition \ref{Riordan}.  

Theorem \ref{JBbasis} is proved using the work of Frenkel and Khovanov \cite{FK}, who studied graphical calculus for the dual canonical basis of tensor products of finite-dimensional  irreducible representations of $U_{q}(sl_{2})$.
More precisely, we define a new basis for $\inv_{U_{q}} (V_{2}^{\otimes n})$, the space of $U_{q}(sl_2)$-invariants of tensor products of the $3$-dimensional irreducible representation $V_{2}$, by inserting copies of the Jones-Wenzl projector in the dual canonical basis studied in \cite{FK}.
This basis is actually unitriangular to the Frenkel-Khovanov basis, see Theorem \ref{modi}.
The result is a graphical description of invariant tensors in terms of Jacobi diagrams ;  
see e.g.  \cite{kuperberg, cascade, westbury} for related graphical approaches to invariant tensors.  
We expect that this result and its possible generalizations could also be interesting from an algebraic point of view.  

\begin{remark}\label{rem:topopipo}
Consider the projection $Z^h$ of the Kontsevich integral $Z$ onto the space $ \mathcal{B}^{t,h}(n)$ of tree Jacobi diagrams labeled by distinct elements of  $\{1,\ldots, n\}$. 
In Proposition 10.6 of \cite{HMa}, Habegger and Masbaum show that, for string links, the leading term of $Z^h$ determines (and is determined by) the first non-vanishing Milnor link-homotopy invariants. 
The non-injectivity of the map $W^h_n$ for $n\ge 5$ tells us that, expectedly, 
this is in general no longer the case for quantum invariant $W\circ Z$ -- 
yet, it is remarkable that it still determines  the first non-vanishing Milnor link-homotopy invariants of length up to $5$.  
On the other hand, since $Z$ extends to a graded isomorphism on the free abelian group generated by string links, surjectivity of the map $W^h_n$ readily implies surjectivity of the linear extension of $W^h_n\circ Z$ (see also Remark \ref{rem:surj}). 
By Theorem \ref{1}, the surjectivity defect is given by $c^{\otimes n}$; it is not hard to check that, for a $2n$-component string link, the coefficient of $c^{\otimes n}$ in $W\circ Z$ is given by a product of linking numbers (this follows from a similar result at the level of the Kontsevich integral $Z$), and is in particular zero for string links with vanishing linking numbers. \\
Similar observations can be made for the universal $sl_2$ invariant, using Theorem 5.5 of \cite{JS}. 
\end{remark}

The rest of this paper is organized in three sections. 
In Section 2 we recall the definitions of Jacobi diagrams and the $sl_2$ weight system, and give a result which in particular implies Lemma \ref{prop:connected}. 
Section 3 introduces Riordan trees and the tree basis of $\inv(sl_2^{\otimes n})$, which are used to prove Theorem \ref{1}. 
Finally, in Section 4 we recall a few elements from the graphical calculus developed by Frenkel and Khovanov, and use it to prove Theorem \ref{JBbasis}.  
 
\begin{acknowledgments}
The authors are indepted to Daniel Tubbenhauer for insightful comments and stimulating discussions. 
They thank Naoya Enomoto for discussions concerning the content of Section \ref{module}, and Rapha\"el Rossignol for writing the code used in Section \ref{Ker}.
They also thank  Kazuo Habiro, Tomotada Ohtsuki and Louis-Hadrien Robert for valuable comments.  
The first author is supported by the French ANR research project ``VasKho'' ANR-11-JS01-00201.  
The second author is  supported by JSPS KAKENHI Grant Number 15K17539.
\end{acknowledgments}

\section{Jacobi diagrams and the $sl_2$ weight system}\label{sec:2}
In this section we give the definitions of the $sl_{2}$ weight system $W$ and proof of Lemma \ref{prop:connected}.  

\subsection{The Lie algebra $sl_{2}$ and its symmetric algebra}

Recall that the Lie algebra $sl_2$ is the 3-dimensional Lie algebra over $\mathbb{Q}$ generated by $h, e,$ and $f$ with Lie bracket
\begin{align*}
[h,e]=2e, \quad [h,f]=-2f, \quad [e,f]=h.
\end{align*}
Let $S=S(sl_{2})$ be the symmetric algebra of $sl_{2}$.
The adjoint action, acting as a derivation, endows $S$, and more generally $S^{ \otimes n }$ for any $n\ge 1$, with a structure of $sl_{2}$-modules.
Note that $sl_2^{ \otimes n }$, the $n$-fold tensor power of $sl_2$, is isomorphic to the subspace of $S^{ \otimes n }$ having degree one in each factor.

We denote by $\inv(S^{ \otimes n })$ and $\inv(sl_2^{ \otimes n })$ the set of invariant tensors of $S^{ \otimes n }$ and $sl_2^{ \otimes n }$, respectively (that is, elements that are mapped to zero when acted on by $h, e,$ and $f$).

\subsection{Jacobi diagrams}\label{sec:jacobi}

A \emph{Jacobi diagram} is a finite unitrivalent graph, such that each trivalent vertex is equipped with a cyclic ordering of its three incident half-edges. 
Each connected component is required to have at least one univalent vertex. 
An \emph{internal edge} of a Jacobi diagram is an edge connecting two trivalent vertices.   
The \emph{degree} of a Jacobi diagram is half its number of  vertices.

In this paper we call a simply connected (not necessary connected) Jacobi diagram a \textit{tree}. 
A tree consisting of a single edge is called a \emph{strut}. 

Let  $\mathcal{B}(n)$ be the completed $\mathbb{Q}$-space spanned by Jacobi diagrams 
whose univalent vertices are labeled by elements of $\{1,...,n\}$, 
subject to the AS and IHX relations shown in Figure \ref{fig:relations}.  
\begin{figure}[!h]
\includegraphics{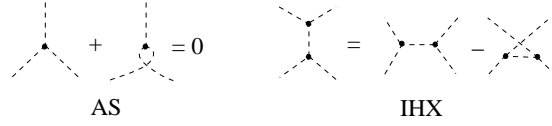}
\caption{The AS and IHX relations. } \label{fig:relations}
\end{figure}
Here completion is given by the degree.
Note that $\mathcal{B}(n)$ has an algebra structure with multiplication given by disjoint union.

Let  $\mathcal{B}^h(n)\subset \mathcal{B}(n)$ denote the subspace generated by Jacobi diagrams labeled by distinct\footnote{The superscript $h$ stands for `homotopy' since, as noted in the introduction, $\mathcal{B}^h(n)$ is the relevant space for link-homotopy invariants of (string) links.  } 
elements in $\{1,\ldots, n\}$. 
Note that $\mathcal{B}^{h}(n)$ is the polynomial algebra on the space $\mathcal{C}^h(n)$ of connected diagrams labeled by distinct elements in $\{1,\ldots,n\}$

As is customary, for each of the spaces defined above we use a subscript $k$ to denote the corresponding subspaces spanned by degree $k$ elements.

We denote by $\mathcal{C}_{n}$ the space of connected trees where each of the labels $1, \ldots, n$ appears exactly once.
It is a well-known fact, easily checked using the AS and IHX relations, that a basis for $\mathcal{C}_{n}$ is given by \emph{linear trees}, 
i.e. connected trees of the form shown in Figure \ref{fig:mTm}, 
where the labels $i_1$ and $i_n$ are two arbitrarily chosen elements of $\{1,\ldots,n\}$, and where 
$i_2,\ldots,i_{n-1}$ are running over all (pairwise distinct) elements of  $\{1,\ldots,n\}\setminus \{i_1,i_n\}$. 
 \begin{figure}[!h]
   \input{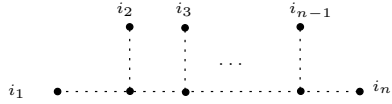}
    \caption{A linear tree Jacobi diagram. }\label{fig:mTm}
 \end{figure}
This shows that $\dim \mathcal{C}_{n}=(n-2)!$, as recalled in the introduction.  

\subsection{The $sl_{2}$ weight system}\label{sec:W}

We now define the \emph{$sl_{2}$ weight system}, which is a  $\mathbb{Q}$-algebra homomorphism
 $$ W\co \mathcal{B}(n) \to \inv(S^{\otimes n}). $$
Recalling that $\mathcal{B}(n)$ is (the completion of) the commutative polynomial algebra on the space of connected diagrams, 
it is enough to define it on the latter. We closely follow \cite[\S 4.3]{JS}.  

We will use the  non-degenerate symmetric bilinear form
 $$ \kappa\co  sl_2\otimes sl_2 \rightarrow \mathbb{Q} $$ 
 given by 
$$ \kappa(h,h)= 2, \quad \kappa(e,f)=1, \quad \kappa(h,e) = \kappa(h,f) = \kappa(e,e) =\kappa(f,f)= 0. $$
The bilinear form $\kappa$ identifies $sl_{2}$ with the dual Lie algebra $sl_{2}^{*}$.  
Note that, under this identification, $\kappa \in  (sl_2^{\otimes 2})^{*}\simeq sl_2^{*}\otimes sl_2^{*}$ 
itself corresponds to the quadratic Casimir tensor 
\begin{align}\label{c}
c=\frac{1}{2}h\otimes h+f\otimes e +e\otimes f \in \inv (sl_{2}^{\otimes 2}),
\end{align}
while the Lie bracket $[ -, - ]\in sl_2^{*} \otimes sl_2^{*} \otimes sl_{2}$  corresponds to the invariant tensor 
\begin{align}\label{csb}
\begin{split}
b &= \sum_{\sigma\in \mathfrak{S}_3} (-1)^{|\sigma |} \sigma (h\otimes e\otimes f)
\\
&=h\otimes e \otimes f + e\otimes f\otimes h + f\otimes h\otimes e 
 -h\otimes f\otimes e -f\otimes e\otimes h -e\otimes h \otimes f.
\end{split}
\end{align}
where $\sigma$ acts by permutation of the factors.

Let $D_{ij}$ be a strut with vertices labeled by $1\leq i, j\leq n$.
Rewriting formally (\ref{c}) as $c=\sum c_1\otimes c_2$, 
we set
$$ W(D_{ij}) = \sum 1\otimes \cdots \otimes c_1 \otimes \cdots  \otimes c_2 \otimes \cdots \otimes 1\in  \inv(S^{ \otimes n}),  $$
where $c_1$  and $c_2$ are at the $i$th and $j$th position, respectively. 

Now, let $m\geq 2.$
For a diagram connected diagram $D\in \mathcal{B}_m(n)$, attach a copy of $b\in \inv(sl_2^{\otimes 3})$ to each trivalent vertex of $D$, 
a copy of $sl_2$ being associated to each of the  $3$ half-edges at the trivalent vertex following the cyclic ordering.   
Each internal edge of $D$ is divided into to half-edges, and we contract the two corresponding copies of $sl_2$ by $\kappa$.
Fixing an arbitrary total order on the set of univalent vertices of $D$, we get in this way an element 
$x_{D}=\sum x_1\otimes \cdots \otimes x_{m+1}$ of $\inv (sl_2^{\otimes m+1})$, the $i$th factor corresponding to the $i$th univalent vertex of $D$. 
We then define $W(D) \in \inv(S^{\otimes n})$ by 
\begin{align}\label{formula}
 W(D)= \sum y_1\otimes \cdots \otimes  y_n,
\end{align}
where $y_j$ is the product of all $x_{i}\in sl_2$ such that the $i$th vertex is labeled by $j$.

It is known that $W$ is well-defined, i.e. is invariant under the AS and IHX relations.
The next lemma, due to Chmutov and Varchenko \cite{CV}, gives another relation satisfied by the $sl_2$ weight system.  
\begin{lemma}\label{lem:CV}
The $sl_2$ weight system $W$ factors through the CV relation below\\[0.1cm]
\begin{center}
\includegraphics{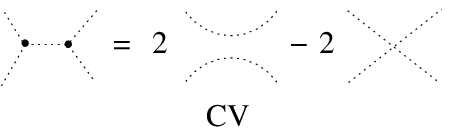}
\end{center}
\end{lemma}
Note that the CV relation is not degree-preserving. 
Note also that this relation might involve diagrams with a circular component : the value of $W$ on such component is set to $W(\bigcirc)=3=\dim sl_2$.  

\begin{remark}\label{rem:target}
It is worth noting here that the restriction of the $sl_2$ weight system to $\mathcal{C}_n$ takes values in $\inv (sl_{2}^{{\otimes n}})$.  
Likewise, the homotopy $sl_2$ weight system, i.e. its restriction to $\mathcal{B}^h(n)$, takes values in 
$\inv (\langle sl_{2}\rangle^{{\otimes n}})$, where $\langle sl_{2}\rangle^{{\otimes n}}=(\mathbb{Q} \oplus sl_{2})^{\otimes n}\subset S(sl_2)^{\otimes n}$ 
is the subspace of tensors having degree \emph{at most one} in each factor. 
\end{remark}

\subsection{The space $\mathcal{B}_{sl_2}(n)$ of $sl_2$-Jacobi diagrams}\label{sec:Bsl2}

In view of Lemma \ref{lem:CV}, it is natural to consider the following space. 
\begin{definition}\label{Bsl2}
The space of $sl_2$-Jacobi diagrams is the quotient space 
 $$ \mathcal{B}_{sl_2}(n) = \mathcal{B}(n)/CV,\bigcirc_3 $$ 
of $\mathcal{B}(n)$ by the ideal generated by the CV relation and the relation $\bigcirc_3$ that maps a circular component to a factor $3$. 
\end{definition}
Note that the algebra structure on $\mathcal{B}(n)$ descends to $\mathcal{B}_{sl_2}(n)$.  
This is however no longer a graded algebra (although one could impose such a structure by considering the number of univalent vertices). 

Since the $sl_2$ weight system factors through $\mathcal{B}_{sl_2}(n)$, it is useful to for our study to get some insight in this space.  
\begin{proposition}\label{prop:Bsl2}
  As an algebra, $\mathcal{B}_{sl_2}(n)$ is generated by (connected) trees.
\end{proposition}
This in particular implies Lemma \ref{prop:connected} stated in the introduction.   

\begin{proof}
It suffices to prove that any connected Jacobi diagram in $\mathcal{B}_{sl_2}(n)$ can be expressed as a combination of trees. 
The proof is by a double induction, on the number of cycles in the diagrams and on the minimal length of the cycles
(the length of a cycle is the number of internal edges contained in it). \\  
Consider a connected diagram $C$ with $k$ cycles, and pick a cycle of minimal length $l$.  
If the cycle has length $l=0$, then the diagram $C$ is a loop, which can be replaced by a coefficient $3$ by the $\bigcirc_3$ relation. 
If $l=1$, then it follows from the AS relation that $C$ is zero. 
Now, if $l\ge2$, we can apply the CV relation at some internal edge of the cycle, which gives  
 \begin{equation} \label{W}
  \textrm{\includegraphics[scale=0.9]{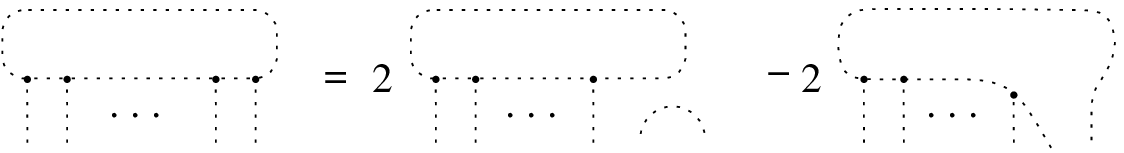}}, 
 \end{equation}
\noindent where the rightmost term is a diagram with $n-1$ cycles, 
and where the middle term has a cycle of length $l-2$.  
We can thus apply (\ref{W}) recursively to reduce the length of this cycle, until we obtain a cycle of length either $1$ or $0$, as above. 
Then $C$ writes as a combination of diagrams with less than $k$ cycles. 
This concludes the proof. 
\end{proof}

\section{Invariant tensors and the homotopy $sl_2$ weight system}\label{sec:3}
In this section we give a basis for $\inv (sl_{2}^{{\otimes n}})$ in terms of Riordan trees, 
and use this basis to prove Theorem \ref{1}. The kernel of the homotopy $sl_2$ weight system is briefly discussed at the end of the section.

\subsection{Tree basis of $\inv (sl_{2}^{{\otimes n}})$}\label{sec:basis}

We now construct a basis for $\inv (sl_{2}^{\otimes n})$, as the image by the $sl_{2}$ weight system of a certain class of connected tree Jacobi diagrams.   
For this, we need a couple extra definitions. 

On one hand, we call a linear tree \emph{ordered} if, in the notation of Figure \ref{fig:mTm}, its labels $i_1,\ldots,i_n$ satisfy $i_1 < i_2 <\ldots <i_n$. 

On the other hand, a \emph{Riordan partition} is a partition of $\{1,\ldots,n\}$ 
into parts that contains at least two elements, and whose convex hulls are disjoint when the points are arranged on a circle. 
For example, $\{\{1,4,5,9,10\},\{2,3\},\{6,7,8\}\}$ is a Riordan partition, as illustrated in Figure \ref{fig:catalan}, 
while $\{\{1,4,6\},\{2,3\},\{5,7,8\}\}$ is not.\footnote{A partition satisfying only the second condition is often called non-crossing.  } 
\begin{figure}[h!]
\includegraphics{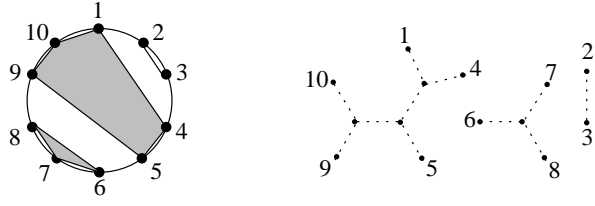}
\caption{The Riordan tree associated to the Riordan partition $\{\{1,4,5,9,10\},\{2,3\},\{6,7,8\}\}$. }\label{fig:catalan}
\end{figure}
The number of Riordan partitions of $\{1,\ldots,n\}$ is given by the Riordan number $R_n$ -- see \cite[\S 3.2]{B}.  

This leads to the following  
\begin{definition}\label{Riordan}
 A \emph{Riordan tree} of order $n$ is an element of $\mathcal{B}^h(n)$ such that
\begin{itemize}
 \item each connected component is an ordered linear tree, 
 \item the partition of $\{1,\ldots,n\}$ induced by its connected components is a Riordan partition. 
\end{itemize}
\end{definition}
See the right-hand side of Figure \ref{fig:catalan} for an example. 
Note that a Riordan partition uniquely determines a Riordan tree ; the number of Riordan trees of order $n$ is thus given by $R_n$.  

\begin{theorem}\label{JBbasis}
The set 
 $$ \mathfrak{I}_n := \{  W(T) \textrm{ ;  $T$ is a Riordan tree of order $n$} \} $$
forms a basis for $\inv(sl_2^{\otimes n})$. 
\end{theorem}
We call this basis the \textit{tree-basis} of $\inv(sl_2^{\otimes n})$.
The proof of Theorem \ref{JBbasis} is postponed to Section \ref{sec:FK}, and is somewhat indirect. 
It uses the graphical calculus for the dual canonical basis for $\inv(V_2^{\otimes n})$ given by Frenkel and Khovanov in \cite{FK}.
Although a more direct proof may exist, we hope that the one given in this paper could be interesting from the representation theory point of view.

\begin{remark}\label{rem:surj}
Theorem \ref{JBbasis} implies immediately 
that the homotopy $sl_2$ weight system $W\co \mathcal{B}^h(n)\rightarrow \inv (\langle sl_{2}\rangle^{{\otimes n}})$ is surjective, 
and Theorem \ref{1} can be regarded as a refinement of this observation. 
(Recall that $\langle sl_{2}\rangle^{{\otimes n}}$ was defined in Remark \ref{rem:target}.)
\end{remark}

\subsection{Proof of Theorem \ref{1}}\label{Res}
The proof of Theorem \ref{1}  (i) is straightforward using Theorem \ref{JBbasis}: 
pick a basis for $\mathcal{C}_n$ in terms of linear trees, as outlined at the end of Section \ref{sec:jacobi},  
and write each basis element, using the CV relation, as the linear combination of basis Riordan trees of order $n$. It then suffices to check that, for $n\le 5$, the matrix obtained in this transformation has rank $(n-2)!$. 
Non-injectivity  for $n\ge 6$ is obvious since the dimension of the target space $\inv (sl_{2}^{\otimes n})$ is smaller than that of the domain $\mathcal{C}_{n}$. 

We now turn to the surjectivity results (ii) and (iii). 
Let $\mathcal{B}^{h}_{Y}(n)\subset \mathcal{B}^{h}(n)$ be the subspace of Jacobi diagrams with at least one trivalent vertex, and let 
$\mathcal{B}^{h}_{U}(n)=\mathcal{B}^{h}(n)\setminus \mathcal{B}^{h}_{Y}(n)$.  
Set
\begin{align*}
\mathfrak{I}_n^{Y}& := \{  W(T) \textrm{ ;  $T$ is a Riordan tree in $\mathcal{B}_{Y}^{h}(n)$} \},
\\
\mathfrak{I}_n^{U} &:= \{  W(T) \textrm{ ;  $T$ is a Riordan tree in $\mathcal{B}_{U}^{h}(n)$} \}. 
\end{align*}
Note that $ \mathfrak{I}_{n}= \mathfrak{I}_{n}^{Y}$ for $n$ odd, 
while  $ \mathfrak{I}_{n}= \mathfrak{I}_{n}^{Y}\cup  \mathfrak{I}_{n}^{U}$ for $n$ even.

Based on Theorem \ref{JBbasis} and this observation, points (ii) and (iii) of Theorem \ref{1} follow from the following two lemmas.
\begin{lemma}\label{c1}
If $T\in \mathcal{B}^{h}_{Y}(n)$, then 
$W(T)\in W\left(\mathcal{C}_n\right)$. 
In particular, $\mathfrak{I}_{n}^{Y}\subset W\left(\mathcal{C}_n\right)$.  
\end{lemma}
For $n\ge 2$ even, let $\cup^{\otimes n}=\coprod_{i=1}^{n/2} D_{2i-1,2i}$ denote the tree diagram made of $n$ struts labeled by $i$ and $i+1$ ($1\le i\le n/2$). 
Note that $W(\cup^{\otimes n})=c^{\otimes n}\in \mathfrak{I}_{n}^{U}$.  
\begin{lemma}\label{c2} 
\begin{itemize}
\item[\rm{(i)}] We have $W(\cup^{\otimes n})\not \equiv 0$ modulo $W\left(\mathcal{C}_n\right)$.
\item[\rm{(ii)}] If  $T\in \mathcal{B}^{h}_{U}(n)$ with $n\geq 4$ even, then $W(T)\equiv W(\cup^{\otimes n})$ modulo $W\left(\mathcal{C}_n\right)$.  
\end{itemize}
\end{lemma}

\begin{proof}[Proof of Lemma \ref{c1}]
Let $T\in \mathcal{B}^h(n)$, containing at least one trivalent vertex, and let $k$ denote the number of connected components of $T$. 
If $k>1$, the equality depicted in Figure \ref{fig:CV_connected} 
\begin{figure}[!h]
  \includegraphics{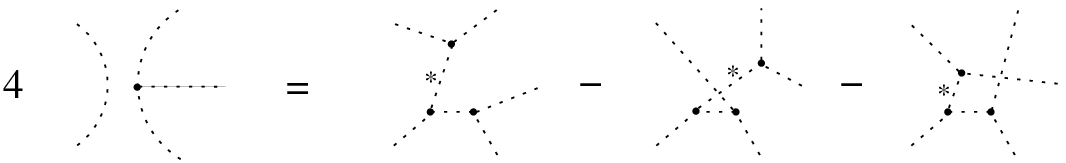}.
  \caption{Relation in $\mathcal{B}_{sl_2}(n)$, given by applying the CV relation at each of the $\ast$-marked edges on the right-hand side.} 
  \label{fig:CV_connected}
\end{figure}
shows how $T$ can be expressed as a combination of tree diagrams with $k-1$ components
in $\mathcal{B}^h_{sl_2}(n)$.  
Since each of these trees contains at least one trivalent vertex, the proof follows by an easy induction on $k$.
\end{proof}

\begin{remark}
Note that the proof applies more generally to the whole space $\mathcal{B}_{sl_2}(n)$. 
More precisely, any Jacobi diagram with at least one trivalent vertex decomposes as a combination 
of connected diagrams in $\mathcal{B}_{sl_2}(n)$.  
Combining this with Proposition \ref{prop:Bsl2}, we have that $\mathcal{B}_{sl_2}(2k+1)$ is generated, as a vector space, 
by connected tree Jacobi diagrams and that 
$\mathcal{B}_{sl_2}(2k)$ is generated by connected trees \emph{and} 
the disjoint union of $n$ struts $\sqcup_{i=1}^k D_{2i-1,2i}$.  
\end{remark}

\begin{proof}[Proof of Lemma \ref{c2}]
To show (ii), note that any $T\in \mathcal{B}^{h}_{U}(n)$ is obtained from  $\cup^{\otimes n}$ by exchanging some labels, 
which implies that $W(T)- W\left(\cup^{\otimes n}\right) \in W\left(\mathcal{B}^{h}_{Y}(n)\right)$ by Lemma \ref{lem:CV}.
Combining this with Lemma \ref{c1}, we have the assertion.

We now prove (i). 
 Consider the $\mathbb{C}$-linear map $\phi \co \inv (sl_{2}^{\otimes n}) \to \mathbb{C}$ defined (using Theorem \ref{JBbasis}) by
$$ \phi(t)= \left\{ \begin{array}{ll}
                 0 & \textrm{ for } \ t\in \mathfrak{I}_{n}^{Y}, \\
                 1 & \textrm{ for } \ t\in \mathfrak{I}_{n}^{\cup}.
                    \end{array} \right. $$

We prove that $W(\mathcal{C}_n)\subset \ker(\phi)$, which implies the assertion. 
It suffices to prove that $W(T)\in  \ker(\phi)$ for a connected tree diagram $T\in \mathcal{C}_n$ ; 
actually, as observed at the end of Section \ref{sec:jacobi}, we may further assume that $T$ is linear\footnote{This extra assumption is not necessary for the proof, 
but makes the arguments simpler to verify.  }.  
Notice that the number $v_T$ of trivalent vertices of $T$ is its degree minus $1$, and that applying the CV relation yields diagrams with $(v_T-2)$ trivalent vertices.  
If the degree of $T$ is odd, then by applying the CV relation repeatedly we obtain 
$$ T = 2^{v_T/2} \sum_{i=1}^{2^{v_T/2}} (-1)^{i} U_i, $$
where $U_i\in \mathcal{B}_{U}^{h}(n)$.  
Although this expression is not unique, this always yields $\phi(T)=0$.                
Now, in the case where $T$ has even degree, successive applications of the CV relation give 
$$ T = 2^{(v_T-1)/2} \sum_{i=1}^{2^{(v_T-1)/2}} (-1)^{i} Y_i, $$
where $Y_i$ has a single trivalent vertex (and $\frac{v_T-1}{2}=\frac{n}{2}-1$ struts). 
We thus obtain $\phi(T)=0$, as desired.  
\end{proof}

\subsection{$\mathfrak{S}_{n}$-module structure}\label{module}

For a partition $\lambda$ of $n$, let $V_{\lambda}$ denote the irreducible representation of $\mathfrak{S}_{n}$ associated to $\lambda$.  
Note that the adjoint representation of $sl_{2}$ corresponds to the vector representation $V$ of $SO(3)$,
and  the invariant part of $sl_{2}^{\otimes n}$ corresponds to the invariant part of $V^{\otimes n}$. 
The tensor powers of the vector representation of $GL(3)$ and its restriction to $SO(3)$ are well-studied classically, using e.g. Schur-Weyl duality or Peter-Weyl Theorem. In particular, we have the following. 
\begin{lemma}\label{sn}
As $\mathfrak{S}_{n}$-modules, we have 
$$ \mathrm{Inv}(sl_{2}^{\otimes n})\simeq \bigoplus V_{\lambda},$$
where the summation is over partitions $\lambda=(\lambda_{1}, \lambda_{2}, \lambda_{3} )$ of $n$ such that each $\lambda_{i}$ is odd or each $\lambda_{i}$ is even, i.e., such that $\lambda_{1}-\lambda_{2}, \lambda_{2}-\lambda_{3}\in 2\mathbb{Z}$.
\end{lemma}

Corollary \ref{cha} follows from Theorem \ref{1} and  Lemma \ref{sn} as follows.
\begin{proof}[Proof of Corollary \ref{cha}]
The fact that $\chi_{\mathrm{ker} (W_n^h)}=\chi_{\mathcal{C}_{n}}-\chi_{\mathrm{Inv}(sl_{2}^{\otimes n})}$ and $\chi_{\mathrm{Im} (W_n^h)}=\chi_{\mathrm{Inv}(sl_{2}^{\otimes n})}$
for $n=2$ or  $n> 2$ odd immediately follows from Theorem \ref{1}.
By Lemma \ref{sn},  the one dimensional representation appearing in the irreducible decomposition of $\mathrm{Inv}(sl_{2}^{\otimes n})$ 
is the trivial representation $U$.  
Thus we have that $\chi_{\mathrm{ker} (W_n^h)}=\chi_{\mathcal{C}_{n}}-\chi_{\mathrm{Inv}(sl_{2}^{\otimes n})}+ \chi_{U}$  
and $\chi_{\mathrm{Im} (W_n^h)}=\chi_{\mathrm{Inv}(sl_{2}^{\otimes n}) }- \chi_{U}$  
for $n\geq 4$ even. 
\end{proof}

The character  $\chi_{\mathcal{C}_{n}}$  is known as follows. 
\begin{proposition}[{Kontsevich \cite[Theorem 3.2]{K}}]\label{Kon}
As a $\mathfrak{S}_{n}$-module, the character of $\mathcal{C}_{n}$  is
{\fontsize{9pt}{0}
\begin{align*}
\chi _{\mathcal{C}_{n}}(1^{n})=(n-2)!, \ \  \chi_{\mathcal{C}_{n}}(1^{1}a^{b})=(b-1)!a^{b-1}\mu(a), \ 
\ \chi_{\mathcal{C}_{n}}(a^{b})=-(b-1)!a^{b-1}\mu(a),
\end{align*}}
and $\chi_{\mathcal{C}_{n}}(\ast)=0$ for other conjugacy classes. Here, $\mu$ is the
M\"obius  function.
\end{proposition}
Thus by Corollary \ref{cha} we can calculate the character $\chi_{\mathrm{ker} (W_n^{h})}$ explicitly. See Figure \ref{fig:tab} for the low degree cases.  
\begin{figure}[!h]
\fontsize{5pt}{0} \selectfont
\begin{picture}(300,90)
\put(0,80){{\small $n=2:$} \quad $\Yvcentermath1 \young(\hfil\hfil)$}
\put(0,60){{\small $n=3:$} \quad $\Yvcentermath1 \young(\hfil,\hfil,\hfil)$}
\put(0,35){{\small $n=4:$} \quad  $\Yvcentermath1 \young(\hfil\hfil,\hfil\hfil)$}
\put(0,10){{\small $n=5:$} \quad $\Yvcentermath1 \young(\hfil\hfil\hfil,\hfil,\hfil)$}
\put(0,-20){{\small $n=6:$}  \quad $\Yvcentermath1 \young(\hfil\hfil\hfil\hfil,\hfil\hfil)  \  { \bigoplus }  \  {\color{magenta}\young(\hfil\hfil\hfil,\hfil,\hfil,\hfil)}  \  { \bigoplus }  \ \young(\hfil\hfil,\hfil\hfil,\hfil\hfil)$}
\put(0,-50){{\small $n=7:$} \quad $ \Yvcentermath1 \young(\hfil\hfil\hfil\hfil\hfil,\hfil,\hfil)  \  { \bigoplus }  \  {\color{magenta} \young(\hfil\hfil\hfil\hfil,\hfil\hfil,\hfil)}  \  { \bigoplus }  \ \young(\hfil\hfil\hfil,\hfil\hfil\hfil,\hfil)  \  { \bigoplus }  \ {\color{magenta}\young(\hfil\hfil\hfil,\hfil\hfil,\hfil,\hfil)}  \  { \bigoplus }  \ {\color{magenta}\young(\hfil\hfil,\hfil\hfil,\hfil,\hfil,\hfil)}$}
\put(0,-80){{\small $n=8:$} \quad $\Yvcentermath1 \young(\hfil\hfil\hfil\hfil\hfil\hfil,\hfill\hfill)  \  { \bigoplus }  \ {\color{magenta}\young(\hfil\hfil\hfil\hfil\hfil,\hfil\hfil,\hfil)}  \  { \bigoplus }  \ {\color{magenta}\young(\hfil\hfil\hfil\hfil\hfil,\hfil,\hfil,\hfil)}  \  { \bigoplus }  \ \young(\hfil\hfil\hfil\hfil,\hfil\hfil\hfil\hfil)  \  { \bigoplus }  \ {\color{magenta}\young(\hfil\hfil\hfil\hfil,\hfil\hfil\hfil,\hfil)}$  $  \  { \bigoplus }  \  {\color{magenta}\young(\hfil\hfil\hfil\hfil,\hfil\hfil,\hfil\hfil)}   \  { \bigoplus }  \  \young(\hfil\hfil\hfil\hfil,\hfil\hfil,\hfil\hfil)$}
\put(35,-120){$ \Yvcentermath1 \  { \bigoplus }  \  {\color{magenta}\young(\hfil\hfil\hfil\hfil,\hfil,\hfil,\hfil,\hfil)}  \  { \bigoplus }  \ {\color{magenta} \young(\hfil\hfil\hfil\hfil,\hfil\hfil,\hfil,\hfil)}  \  { \bigoplus }  \ {\color{magenta}\young(\hfil\hfil\hfil\hfil,\hfil,\hfil,\hfil,\hfil)}  \  { \bigoplus }  \ {\color{magenta}\young(\hfil\hfil\hfil,\hfil\hfil\hfil,\hfil,\hfil)}  \  { \bigoplus }  \  {\color{magenta}\young(\hfil\hfil\hfil,\hfil\hfil\hfil,\hfil,\hfil)}  \  { \bigoplus }  \ {\color{magenta}\young(\hfil\hfil\hfil,\hfil\hfil,\hfil\hfil,\hfil)}  \  { \bigoplus }  \ {\color{magenta}\young(\hfil\hfil\hfil,\hfil\hfil,\hfil,\hfil,\hfil)}  \  { \bigoplus }  \ {\color{magenta}\young(\hfil\hfil,\hfil\hfil,\hfil\hfil,\hfil\hfil)}  \  { \bigoplus }  \ {\color{magenta}\young(\hfil\hfil,\hfil\hfil,\hfil,\hfil,\hfil,\hfil)}$}
\end{picture}
\vspace{130pt}

\caption{Irreducible decompositions of $\mathcal{C}_{n}$, $2\leq n\leq 8$,  as $\mathfrak{S}_{n}$-modules. The red components are in the kernel of $W_{n}^{h}$.}\label{fig:tab}
\end{figure}

\subsection{Generating the kernel}\label{Ker}

It follows from Theorem \ref{1} that the dimension of the kernel of the weight system map $W_k$ is given by 
 $k! - R_k + \frac{1+ (-1)^k}{2}$.  
In this short section, we investigate some typical elements of this kernel. More precisely, we consider 
\emph{$1$-loop relators} of degree $k$, which are linear combinations of elements of $\mathcal{C}_{k+1}$ of the form 
 $$ L_1 - L_2 - R_1 + R_2, $$ 
where $L_1, L_2, R_1, R_2$ are 
degree $k$ tree Jacobi diagrams 
as shown in Figure \ref{fig:1loop}. 
\begin{figure}[!h]
  \input{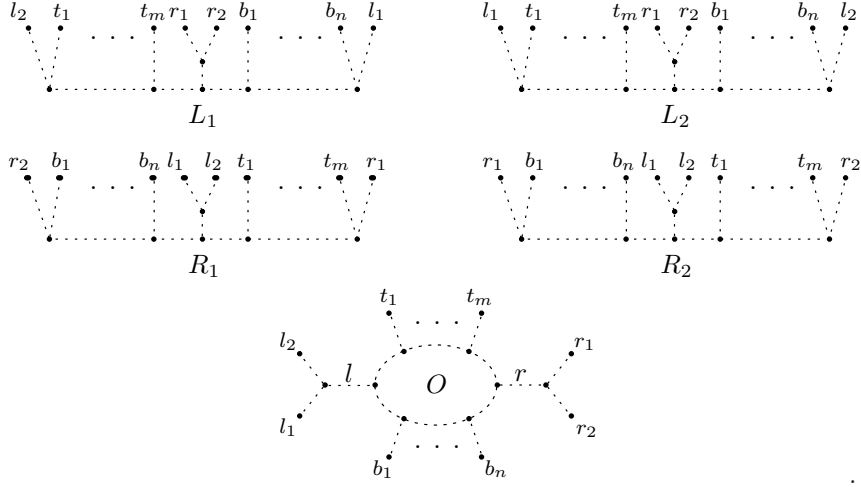}.
  \caption{The diagrams $L_1, L_2, R_1, R_2$ and $O$ ; here $m,n\ge 0$ are such that $k=(m+n+2)/2$. } 
  \label{fig:1loop}
\end{figure}

Let us explain why these are indeed mapped to zero by $W_k$. 
Denote by $O$ the element of $\mathcal{B}^h_{k+1}(k+1)$ represented in Figure \ref{fig:1loop}.    
We call such an element a \emph{$2$-forked wheel}.   
Now, by applying the CV relation at the internal edge $l$ of $O$ (see the figure), we have that 
 $$ W_k(O) = 2 W_k(L_1) - 2W_k(L_2), $$
while applying CV at internal edge $r$ yields
 $$ W_k(O) = 2 W_k(R_1) - 2W_k(R_2), $$
thus showing that $L_1 - L_2 - R_1 + R_2$ is in the kernel of $W_k$. 

Notice that, in degree $\le 5$, all $1$-loop relators are trivial, which agrees with the fact that the weight system map is injective. 
Computations performed using a code in \texttt{Scilab} 
allowed us to check that,  up to degree $k=8$,  the kernel of the weight system map $W_k$ is generated by $1$-loop relators of degree $k$.\footnote{The authors are indebted to Rapha\"el Rossignol for writing this code. } 
It would be interesting to see up to what degree this statement still holds, and what are the additional kernel elements when it doesn't.  

\section{The dual canonical basis and the $sl_{2}$ weight system}\label{sec:FK}

In this section, we review the graphical calculus used by Frenkel and Khovanov in \cite{FK} 
to describe tensor products of finite-dimensional irreducible representations of quantum group $U_q(sl_2)$. 
This graphical calculus for invariant tensors appeared originally in the work of Rumer, Teller and Weyl \cite{Weyl1932}, 
and was later adapted to the quantum setting in \cite{FK}. 

More precisely, we first recall in Section \ref{sec:41} some basic facts on $U_q(sl_2)$ and its representations, 
in Section \ref{FK} we recall the graphical calculus for the dual canonical basis 
for invariant tensors of $3$-dimensional irreducible representations of  $sl_2$,  
and in Section \ref{JB} we show that a simple modification of this basis 
is well-behaved with respect to the universal $sl_2$ weight system. 

\subsection{Quantum group $U_{q}(sl_{2})$ and finite-dimensional irreducible representations}\label{sec:41}

Let $U_{q}=U_{q}(sl_{2})$ be the algebra over $\mathbb{C}(q)$ with generators  $K, K^{{-1}}, E, F$ and relations
$$ KK^{-1}=K^{-1}K=1,
\quad 
KE=q^{2}EK,
\quad
KF=q^{-2}FK,
\quad
EF-FE=\frac{K-K^{-1}}{q-q^{-1}}, 
$$
\noindent for $q$ a non-zero complex indeterminate. 

For $n\ge 0$, denote by $V_n$ the fundamental $(n+1)$-dimensional irreducible representation of $U_q$, with basis 
$$ \{ v_i\textrm{ ; $-n\le i\le n$, $i=n$ (mod $2$)} \} $$ 
such that the action of $U_q$ is given by 
$$ Ev_i = \left[\frac{n+i+2}{2} \right] v_{i+2}, \quad Fv_i = \left[\frac{n-i+2}{2} \right] v_{i-2},\quad K^{\pm 1}v_i = q^{\pm 1}v_{i} $$
where $\left[ m \right]=(q^{m}-q^{-m})/(q-q^{-1})$ and  $v_{n+2}=v_{-n-2}=0$.

Let $\langle~ ,~ \rangle\co V_{n} \otimes V_{n}\to \mathbb{C}(q)$ be the symmetric bilinear pairing defined by 
$$ \langle v_{n-2k},v_{n-2l}\rangle =\frac{[n]!}{[k]![n-k]!} \delta_{k,l}\textrm{ ; }0\le k,l\le n,  $$
where $[k]!:=\prod_{i\le k} [i]$, 
and let $\{ v^i\textrm{ ; $-n\le i\le n$, $i=n$ (mod $2$)} \}$ be the dual basis with respect to this pairing.
In particular, for $n=1$, the dual basis is simply given by $v^i=v_i$ ($i=\pm 1$), while for $n=2$, we have 
$v^2=v_2$, $v^0=\frac{1}{[2]}v_0$ and $v^{-2}=v_{-2}$.  

We also define the bilinear pairing  $\langle~ ,~ \rangle$ of $V_{n_1}\otimes\ldots \otimes V_{n_m}$ and $(V_{n_1}\otimes\ldots \otimes V_{n_m})^{*}=V_{n_m}\otimes\ldots \otimes V_{n_1}$
by\footnote{
The action of $U_{q}$ on tensor powers of irreducible representations 
is defined via the comultiplication map $\Delta$ in the Hopf algebra structure of $U_{q}$ ; the dual action with respect to $\langle~ ,~ \rangle$ is likewise given by $u(x\otimes y)=\bar \Delta(u)(x\otimes y)$, where $\bar\Delta (u)=(\sigma \otimes \sigma) \Delta (\sigma(u))$ with the bar involution 
$\sigma \co U_{q} \to U_{q}$. See e.g. \cite[Chap.~3]{jantzen} or \cite[\S~1]{FK} for the details. }
$$\langle v_{k_1}\otimes\ldots \otimes v_{k_m}, v^{k_m'}\otimes\ldots \otimes v^{k_1'} \rangle=\prod_{i=1}^m \delta_{k_i,k_i'}.  $$

We refer the reader to Chapters 2 and 3 of the book \cite{jantzen} for a much more detailed treatment of this subject.

\subsection{Graphical representations of the dual canonical basis for invariants tensor products}\label{FK}

In what follows, we will only deal with $1$ and $2$-dimensional representations, which is sufficient for the purpose of this paper. 
We thus only give a very partial overview of the work in \cite{FK}, where we refer the reader for further reading.   
We will mostly follow the notation of \cite{FK}.  

Let $\delta_1\co \mathbb{C}\rightarrow V_1\otimes  V_1$ denote the map defined by 
$$ \delta_1(\mathbf{1})=v^{1}\otimes v^{-1}-q^{-1} v^{-1}\otimes v^1. $$
In \cite[Thm. 1.9]{FK}, Frenkel and Khovanov showed that the intersection of the dual canonical basis of $V_1^{\otimes 2m}$ and the space $ \inv_{U_{q}}(V_1^{\otimes 2m})$ of invariant tensors forms a basis of $ \inv_{U_{q}}(V_1^{\otimes 2m})$: 
$$ \{ (\delta_1)^{2(m-1)}_{i_{m-1}}(\delta_1)^{2(m-2)}_{i_{m-2}}\cdots  (\delta_1)^{2}_{i_1} (\delta_1).\mathbf{1} \textrm{ ; $0\le i_j\le j$ for each index $1\le j\le m-1$} \},  $$
where $(\delta_1)_l^k: V_1^{\otimes k}\rightarrow V_1^{\otimes k+2}$ is 
defined by $(\delta_1)_l^k=\mathbf{1}^{\otimes l}\otimes \delta_1\otimes \mathbf{1}^{\otimes (k-l)}$ ($0\le l\le k$).

Graphically, $V_1^{\otimes 2m}$ is represented by $2m$ fixed points on the $x$-axis of the real plane, and an element of the dual canonical basis of $\inv_{U_{q}}(V_1^{\otimes 2m})$ is represented by a union of $m$ non-intersecting arcs embedded in the lower half-plane and connecting these points, each arc corresponding to a copy of the map $\delta_1$. 
For example, the dual canonical basis of $\inv_{U_{q}}(V_1^{\otimes 4})$ consists of two elements $(\delta_1)^{2}_{0}.(\delta_1).\mathbf{1}$ and $(\delta_1)^{2}_{1}.(\delta_1).\mathbf{1}$, which are  represented by the two diagrams $D_1$ and $D_2$ in Figure \ref{fig:basis0}, respectively.  
\begin{figure}[h!]
\includegraphics{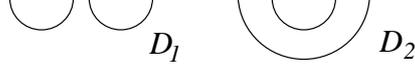}
\caption{FK diagrams representing the dual canonical basis of $ \inv_{U_{q}}(V_1^{\otimes 2})$.  }\label{fig:basis0}
\end{figure}

Now, it follows from \cite[Thm. 1.11]{FK}  that these basis elements induce a  basis $\mathcal{B}^0_{m}$ for $ \inv_{U_{q}}(V_2^{\otimes m})$, by taking their image under $\pi_2^{\otimes m}$, where $\pi_2\co V_1\otimes  V_1\rightarrow V_2$ is defined by 
 \begin{align*}
 \pi_2(v^1\otimes v^1)&=v^{2}, & \pi_2(v^{-1}\otimes v^{-1})&=v^{-2}, 
\\
 \pi_2(v^{1}\otimes v^{-1})&=q^{-1}v^{0}, & \pi_2(v^{-1}\otimes v^{1})&=v^{0}.
\end{align*}
The map $\pi_2$ is graphically represented by a box with two incident points (corresponding to the two copies of $V_1$) on its lower horizontal edge, see  Figure \ref{fig:basis}.  

Since $\pi_2\circ \delta_1=0$, a diagram containing a box whose incident points are connected by an arc is equal to zero. 
If there is no such  box, then this defines a non-trivial element of $\inv_{U_{q}} (V_2^{\otimes m})$. \\
In summary, an element of the dual canonical basis $\mathcal{B}^0_{m}$ of $\inv_{U_{q}} (V_2^{\otimes m})$ is graphically incarnated by $m$ horizontally aligned boxes, whose incident edges are connected by $m$ non-intersecting arcs, such that each arc is incident to two distinct boxes. 
\begin{remark}
Arranging the $n$ boxes on a circle, FK diagrams for elements of $\mathcal{B}^0_{m}$ 
naturally appear to be in one-to-one correspondence with (convex hulls of) Riordan partitions of $\{1,\ldots,n\}$.  
This agrees with the fact that the dimension of $\inv(V_2^{\otimes m})$ is given by the Riordan number $R_m$. 
\end{remark}

\begin{example}\label{ex:cb}
We conclude with a couple of examples. 
For $m=2$,  $\inv_{ U_{q}}(V_2^{\otimes 2})$ is spanned by $D_c$ in Figure \ref{fig:basis}, which represents the element 
\begin{eqnarray*}
\tilde{c}:& = & (\pi_2\otimes \pi_2)(\delta_1)^{2}_{1}.(\delta_1).\mathbf{1}  \\
& = & (\pi_2\otimes \pi_2)\left( v^{1}\otimes v^{1}\otimes v^{-1}\otimes v^{-1}  -  q^{-1} v^{-1}\otimes v^1\otimes v^{-1}\otimes v^{1}\right. \\
   &      &  \left.-  q^{-1} v^{1}\otimes v^{-1} \otimes v^{1}\otimes v^{-1}  +  q^{-2} v^{-1}\otimes v^{-1} \otimes v^{1}\otimes v^1 \right) \\
   &  =  &  v^{2}\otimes v^{-2} - (q^{-1}+q^{-3}) v^{0}\otimes v^{0}  + q^{-2} v^{-2}\otimes v^{2}.
\end{eqnarray*}
\begin{figure}[h!]
\includegraphics{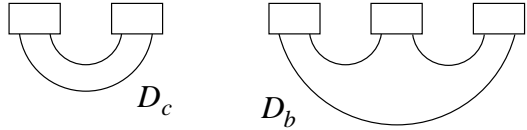}
\caption{FK diagrams representing the dual canonical bases for $ \inv_{U_{q}}(V_2^{\otimes 2})$ and  $ \inv_{U_{q}}(V_2^{\otimes 3})$.  }\label{fig:basis}
\end{figure}

Similarly, $ \inv_{U_{q}} (V_2^{\otimes 3})$ has dimension $1$ with basis given by the diagram $D_b$ represented in Figure \ref{fig:basis}. 
We leave it to the reader to verify that this diagram represents the element
$$ \tilde {b}:=v^{2}\otimes v^{0}\otimes v^{-2} + q^{-2} v^{0}\otimes v^{-2}\otimes v^{2} + q^{-2} v^{-2}\otimes v^{2}\otimes v^{0} + q^{-5} v^{0}\otimes v^{0}\otimes v^{0} $$
$$ - q^{-2} v^{2}\otimes v^{-2}\otimes v^{0} - q^{-2} v^{0}\otimes v^{2}\otimes v^{-2} - q^{-2} v^{-2}\otimes v^{0}\otimes v^{2} - q^{-1} v^{0}\otimes v^{0}\otimes v^{0}. $$
\end{example}

In the rest of this paper, we will use the term \emph{FK diagrams} to refer to this graphical calculus of Frenkel and Khovanov, 
and we will consider such diagrams up to planar isotopy (fixing only the $m$ boundary boxes corresponding to the $m$ copies of $V_2$ in $\inv(V_2^{\otimes m})$). 
We will also often blur the distinction between an invariant tensor and the FK diagram representing it.   

\subsection{The tree basis of $\inv_{U_{q}}(V_{2}^{\otimes n})$}\label{JB}
In this section, we modify the dual canonical basis $\mathcal{B}^0_{m}$ of $\inv_{U_{q}} (V_{2}^{{\otimes n}})$ recalled above and prove that, at $q=1$, this new basis corresponds to the tree basis of $\inv (sl_{2}^{{\otimes n}})$ defined in Section \ref{sec:basis}.

The only new ingredient is the \emph{Jones-Wenzl projector}
$$p_{2} \co  V_1\otimes  V_1 \rightarrow V_1\otimes  V_1$$ 
defined by
 \begin{align*}
 p_2(v^1\otimes v^1)&=v^1\otimes v^1, & p_2(v^{1}\otimes v^{-1})&=\frac{1}{[2]}\left(q^{-1} v^{1}\otimes v^{-1}+v^{-1}\otimes v^{1}\right), 
\\
 p_2(v^{-1}\otimes v^{-1})&=v^{-1}\otimes v^{-1}, & p_2(v^{-1}\otimes v^{1})&=\frac{1}{[2]}\left( v^{1}\otimes v^{-1}+qv^{-1}\otimes v^{1} \right).
\end{align*}
See Figure \ref{fig:CV2} for a graphical definition. 
\begin{figure}[h!]
\includegraphics{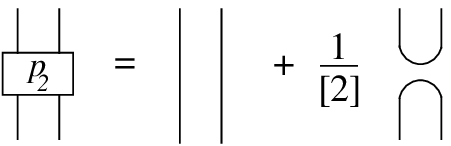}
\caption{The Jones-Wenzl projector $p_{2}\in \textrm{End}(V_1^{\otimes 2})$.  }\label{fig:CV2}
\end{figure}

Let $T$ be a Riordan tree of order $n$. 
We now define two elements $f^0(T)$ and $f(T)$ of $\inv_{U_{q}}(V_2^{\otimes n})$ using the graphical calculus introduced in the previous section. 
Consider a proper embedding $i(T)$ of $T$ in the lower-half plane, such that the $j$-labeled vertex is sent to the point $(j;0)$ 
and such that the cyclic ordering at each trivalent vertex agrees with the orientation of the plane. 
An example is given in Figure \ref{fig:ex_FK1}. 
Note that the Riordan property ensures that such an embedding exists. 
\begin{figure}[h!]
\includegraphics{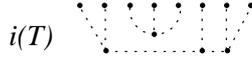}
\caption{The embedding $i(T)$ for the Riordan partition $\{ \{1,2,6,7,8\};\{3,4,5\} \}$.  }\label{fig:ex_FK1}
\end{figure}

We first describe the diagram defining $f^0(T)\in \inv_{U_{q}}(V_2^{\otimes n})$. 
First, replace each point $(j;0)$ by a box representing a copy of $V_2$ ($1\le j\le n$). Next, consider an annular neighborhood of $i(T)$ in the lower-half plane ; 
the boundary of this neighborhood is a collection of disjoint arcs connecting the $n$ boxes, thus providing an FK diagram for $f^0(T)$.  See Figure \ref{fig:ex_FK}.  
\begin{figure}[h!]
\includegraphics{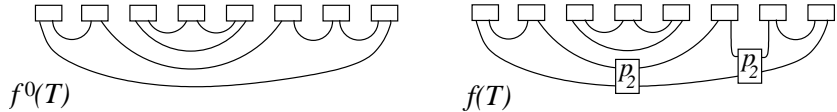}
\caption{The FK diagrams for $f^0(T)$ and $f(T)$, for the Riordan partition $\{ \{1,2,6,7,8\};\{3,4,5\} \}$.  }\label{fig:ex_FK}
\end{figure}
Note that we have the following reformulation for the dual canonical basis of Frenkel--Khovanov:  
 $$ \mathfrak{B}^{0}_n := \{  f^0(T) \textrm{ ;  $f$ is a Riordan tree of order $n$} \}. $$
Now, to obtain the diagram defining $f(T)$  we simply insert a copy of the Jones-Wenzl projector $p_2$ in the pairs of arcs of $f^0(T)$ 
induced by each internal edge of $T$ 
in the above procedure -  see the example of Figure \ref{fig:ex_FK}.  

We have 
\begin{theorem}\label{modi}
The set 
 $$ \mathfrak{B}^{{JW}}_n := \{  f(T) \textrm{ ;  $f$ is a Riordan tree of order $n$} \} $$
forms a basis for $\inv_{U_{q}}(V_2^{\otimes n})$. 
\end{theorem}
\begin{proof}
Since there is a natural one-to-one correspondence between the set $\mathfrak{B}^{JW}_n$ and the basis $\mathfrak{B}^0_n$, 
it is enough to prove the independency of the elements in $\mathfrak{B}^{JW}_n$. 

So suppose that 
$$\sum_{T\in \mathbf{Rio}_n}  \alpha_T f(T) =0, $$ 
where the sum runs over the set $\mathbf{Rio}_n$ of Riordan trees of order $n$, and where $\alpha_T\in \mathbb{C}$. 
Using the formula for the Jones-Wenzl projector $p_2$ given by Figure \ref{fig:CV2}, one can express each $f(T)$ as a linear combination 
 $$ f(T) = f^0(T) + \sum_{T'\subset T} \frac{1}{[2]^{i_T-i_{T'}}} f^0(T'),$$
where the sum runs over all subtrees $T'$ obtained from $T$ by deleting internal edges, and where 
$i_t$ denotes the number of internal edges of a Riordan tree $t\in \mathbf{Rio}_n$.    
By substituting this identity in $\sum_{T\in \mathbf{Rio}_n}  \alpha_T f(T) =0$, we have that there exists complex numbers $\alpha'_T\in \mathbb{C}$ such that $\sum_{f\in \mathbf{Rio}_n}   \alpha'_T f^0(T) =0$, and a lower triangular matrix $A$ whose diagonal entries are all $1$ such that 
$( \alpha'_{T_{1}}, \ldots,   \alpha'_{T_{l}})^{t}=A( \alpha_{T_{1}}, \ldots,   \alpha_{T_{l}})^{t}$ for a suitably chosen order $\{T_1,\ldots ,T_l\}$ on $\mathbf{Rio}_n$. 
Since $\mathfrak{B}^{0}_n$ is a basis of $\inv_{U_{q}}(V_2^{\otimes n})$, we have $( \alpha'_{f_{1}}, \ldots,   \alpha'_{f_{l}})=0$, which implies that $\alpha_T=0$ for all $T\in \mathbf{Rio}_n$. 
This concludes the proof. 
\end{proof}

It turns out that this simple modification of the dual canonical basis of $\inv_{U_{q}}(V_2^{\otimes n})$ 
is directly related to the tree basis introduced in Section \ref{sec:basis}, as we now explain. 

Let $\rho\co \inv_{U_{q}} (V_{2}^{\otimes n}) \to \inv(sl_{2}^{\otimes n})$ be  the $\mathbb{C}$-linear map such that 
$$
\rho(q)=1, \quad \rho(v^{0})=\frac{1}{2}h, \quad  \rho(v^{2})=-e, \quad  \rho(v^{-2})=f.
$$
\begin{proposition}\label{prop:rho}
Let $T$ be a Riordan tree. If $\textrm{deg}(T)$ and $\textrm{tri}(T)$ denote the degree and number of trivalent vertices of $T$ respectively, then we have 
$$\rho\left( f(T) \right)=\frac{(-1)^{\textrm{deg}(T)}}{2^{\textrm{tri}(T)}}W(T). $$
\end{proposition}

It follows immediately that the tree basis of Section \ref{sec:basis} indeed is a basis for $\inv(sl_{2}^{\otimes n})$, as claimed in Theorem \ref{JBbasis}.  

\begin{proof}
The assertion follows essentially from the definitions. 
To see this, let us slightly reformulate the definition of $f(T)$, still in terms of FK diagrams but in a spirit that is closer to that of $W(T)$. 
For each strut component of $i(T)$, pick a copy of the diagram $D_c$ of Figure \ref{fig:basis}, and take a copy of the diagram $D_b$ for each trivalent vertex so that  
a copy of $V_2$ is associated to each of the incident half-edges following the cyclic ordering.  
For each internal edge of $i(T)$, we contract the two corresponding copies of $V_2$ by the map 
$ \varepsilon_{2}\co V_2\otimes V_2\rightarrow \mathbb{C}$  
defined by 
\begin{align*}
\varepsilon_{2}(v^{2}\otimes v^{-2})&=q^{2}, 
& 
\varepsilon_{2}(v^{0}\otimes v^{0})&=-\frac{1}{q^{-1}+q^{-3}}, 
\\
\varepsilon_{2}(v^{-2}\otimes v^{2})&=1, 
& 
\varepsilon_{2}(v^{i}\otimes v^{j})&=0, \quad \text{if} \quad i+j\not= 0.
\end{align*} 
As observed in \cite{FK}, we have the identity 
$$ \varepsilon_{2}\circ (\pi_{2}\otimes \pi_{2})=\varepsilon_{1}\circ (1\otimes \varepsilon_{1}\otimes 1)\circ (p_{2}\otimes p_{2}), $$
where $\varepsilon_{1}\co V_1\otimes  V_1 \rightarrow V_{0}$ is defined by 
$$ \varepsilon_{1}(v^{1}\otimes v^{-1})=-q \quad;\quad \varepsilon_{1}(v^{-1}\otimes v^{1})=1\quad;\quad 
\varepsilon_{1}(v^{1}\otimes v^{1})=\varepsilon_{1}(v^{-1}\otimes v^{-1}) =0. $$
This formula, as illustrated in Figure \ref{fig:e2} below, simply means that the map $\varepsilon_2$ is the insertion of a copy of $p_2$ at each internal edge. (Recall that $p_2$ is a projector, i.e. $p_2\circ p_2=p_2$.)  
\begin{figure}[h!]
\includegraphics{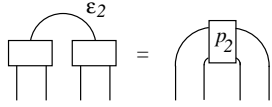}
\caption{Graphical definition of the contraction map $\varepsilon_2$.  }\label{fig:e2}
\end{figure}

So applying $\varepsilon_2$ in this way yields precisely the FK diagram for $f(T)$, where the box corresponding to the $i$-labeled vertex represents the $i$th copy of $V_2$.  
This is illustrated on an example in Figure \ref{fig:tree}.  
\begin{figure}[h!]
\includegraphics{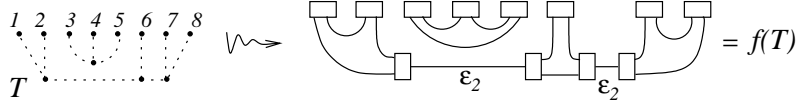}
\caption{Reformulating $f(T)$, for the Riordan partition $\{ \{1,2,6,7,8\};\{3,4,5\} \}$.  }\label{fig:tree}
\end{figure}

Now, it remains to observe that the elements $\tilde{c}$ and $\tilde{b}$, defined in Example \ref{ex:cb} and represented by the diagrams $D_c$ and $D_b$ respectively, correspond to the elements $c$ and $b$ of Equations (\ref{c}) and (\ref{csb}) via the map $\rho$ as follows:   
\begin{equation}\label{eq1}
 \rho( \tilde{c} ) = -c
\end{equation}
and 
\begin{equation}\label{eq2}
 \rho(\tilde{b})  = \frac{1}{2} b, 
\end{equation}
and that the contraction maps $\kappa$ and $\varepsilon_2$, used in the definitions of $W(T)$ and $f(T)$ respectively, are related by 
\begin{equation}\label{eq3}
(\varepsilon_2)_{\vert q=1} = -\kappa\circ \rho. 
\end{equation}
Notice in particular that the $\dfrac{1}{2^{\textrm{tri}(T)}}$ coefficient in the statement comes from the application of (\ref{eq2}) at each trivalent vertex, 
while the sign $(-1)^{\textrm{deg}(T)}$ is given by applying (\ref{eq1}) at each strut component (which has degree $1$), and (\ref{eq3}) at each internal edge (since the degree of a linear tree is the number of internal edges plus $2$). 
This concludes the proof. 
\end{proof}

\bibliographystyle{plain}
\bibliography{paper}
\end{document}